\documentclass[a4paper, 12pt]{article}

\usepackage{amsmath}
\usepackage{amsthm}
\usepackage{amsfonts}
\usepackage{amssymb}
\usepackage{enumerate}
\usepackage{mathrsfs}
\usepackage{mathtools}
\usepackage{tikz}
\usepackage{float} 
\usepackage{xcolor}
\definecolor{gris}{rgb}{.5,.5,.5}

\usetikzlibrary{arrows}
\usetikzlibrary{decorations.pathreplacing}
\usepackage[pagebackref=true, bookmarksopen=true,colorlinks=true, linkcolor=red,citecolor=blue]{hyperref}

\usepackage[ruled]{algorithm2e}

\newtheorem{construction}{Construction}[section]
\newtheorem{theorem}[construction]{Theorem}

\newtheorem{conjecture} {Conjecture}
\newtheorem{corollary} [construction]{Corollary}

\newtheorem{definition} [construction]{Definition}

\newtheorem{lemma} [construction]{Lemma}

\newcommand{\rank}{\operatorname{rank}}
\newcommand{\mnull}{\operatorname{null}}

\begin{document}

\title{2-switch: transition and stability on graphs and forests}


\author{Daniel A. Jaume\footnote{djaume@unsl.edu.ar}\phantom{i},
Adri\'{a}n Pastine\footnote{adrian.pastine.tag@gmail.com}\phantom{i},
and Victor Nicolas Schvöllner\footnote{victor.schvollner.tag@gmail.com}\\
Instituto de Matem\'atica Aplicada San Luis (UNSL-CONICET),\\ 
Universidad Nacional de San Luis\\ 
Av. Italia 1556, C.P.: D5702 BLX\\ 
San Luis, Argentina}

\maketitle

\begin{abstract} 
	Given any two forests with the same degree sequence, we show in an algorithmic way that one can be transformed into the other by a sequence of 2-switches in such a way that all the intermediate graphs of the transformation are forests. We also prove that the 2-switch operation perturbs minimally some well-known integer parameters in families of graphs with the same degree sequence. Then, we apply these results to conclude that the studied parameters have the interval property on those families.
\end{abstract}


\section{Introduction}

Every graph $G=(V,E)$ in the present article is finite, simple, undirected and labeled. We use \(|G|\) and $\|G\|$ to denote the order of \(G\) (i.e. its number of vertices) and the size of $G$ (i.e., the cardinality of $E$) respectively. Unless stated otherwise, we always assume that the set of vertices of $G$ is a subset of $[n]:=\{1,\ldots  ,n\}$, for some $n$. The set of all graphs is denoted by $\mathcal{G}$, and the set of all graphs of order \(n\) with vertex set \([n]\) is denoted by \(\mathcal{G}_{n}\). When there may be ambiguity we use  \(V(G)\) and \(E(G)\) to denote the vertex set and the edge set of $G$, otherwise we just use $V$ and $E$. Vertex adjacency is denoted by $x\sim y$, and we denote the edge $xy$ (i.e., we say that $xy\in E$). If $x$ and $y$ are two vertices of a tree $T$, the sequence of vertices  $(x\dots y)$  symbolizes the (unique) path from $x$ to $y$, or between $x$ and $y$, in $T$. If necessary, more than two vertices of the path can be shown through this notation, for example, we can write $xa\dots b\dots cdy$. The number of connected components of a graph \(G\) is denoted by  \(\kappa(G)\). The complementary graph of a graph \(G=(V,E)\), denoted by \(G^{c}\), is the graph with the same set of vertices \(V\) and \(ab \in E(G^{c})\) if and only if \(ab \notin E(G)\). The subgraph of \(G\) obtained by deleting vertex \(v\) is denoted by \(G-v\). Similarly, \(G-e\) is the subgraph of \(G\) obtained by deleting edge \(e\), \(G+e\) or \(G+ab\) is the graph obtained by adding an edge to \(G\). If $W$ is a set of vertices (edges) of a graph $G$, $G-W$ denotes the subgraph obtained by deleting the vertices (edges) in $W$.

The degree sequence of a graph $G$ with vertex set $V(G)=[n]$ is the sequence $s(G)=(d_{1}, \dots, d_{n})$, where $d_i$ is the degree of vertex $i$.
A sequence $s=(d_{1}, \dots, d_{n})$ is graphical if there is a graph such that $s$ is its degree sequence.

Let \(s=(d_{1},\dots,d_{n})\) be a graphical sequence, the set of all graphs \(G\) with  degree sequence \(s\)  is denoted by  $\mathcal{G}(s)$. Similarly, by $\mathcal{F}(s)$ we denote all the forests \(F\) with  degree sequence \(s\). 

One of the most studied problems in the literature, in regards  $\mathcal{G}(s)$ and  $\mathcal{F}(s)$ is, given a graph parameter (clique number, domination number, matching number, etc.), finding the minimum and maximum values for the parameter in the family, see \cite{BockRat,GHR1,GHR2,wang2008extremal, zhang2008laplacian,zhang2013number}.
Another interesting problem is deciding which values between the minimum and the maximum can be realized by a graph in the family, see \cite{kurnosov2020set,Rao}.

Let $G$ be a graph containing four distinct vertices $a,b,c,d$ such that $ab,cd\in E$ and $ac,bd\notin E$. The process of deleting the edges $ab$ and $cd$ from $G$ and adding $ac$ and $bd$ to $G$ is referred to as a 2-switch in $G$, this is a classical operation, see \cite{Berge, Chartrand1996}. If $G'$ is the graph obtained from $G$ by a 2-switch, it is straightforward to check that $G$ and $G'$ have the same degree sequence. In other words, this operation preserves the  degree sequence.

An important fact about degree sequences is that, given two graphs with the same degree sequence,
one can be obtained from the other by applying successive 2-switches.

\begin{theorem}\label{BergeTeo}
	If $G,H\in \mathcal{G}(s)$, there exists a 2-switch sequence transforming $G$ into $H$. 
\end{theorem}

Theorem \ref{BergeTeo} appears throughout the literature, although its earliest reference appears  most likely in \cite{Berge}. 
A similar result is known for bipartite graphs with a given bipartite degree sequences, where the bipartite degree
sequence of a graph with bipartition $(A,B)$ is the pair $(d_A,d_B)$ containing the degrees of the vertices in $A$ and $B$ respectively. This result follows from a proof of a version of the Havel-Hakimi Theorem (\cite{Hakimi}, \cite{Havel}) for bipartite graphs (see \cite{West}). 

Since Theorem \ref{BergeTeo} assures the existence of a 2-switch sequence transforming $G$ into $H$, 
a natural question to ask is how short can the sequence be. This is studied in \cite{BeregIto},
where they  obtain the length of the shortest possible 2-switch sequence in terms of walks alternating between edges of $G$ and edges of $H$
in the symmetric difference graph between $G$ and $H$. A correct statement 
of their result would need us to introduce some notation that would not be use in this paper,
in lieu of it we provide a weaker statement that is a trivial corollary of their result.
\begin{theorem}
Let $G,H\in \mathcal{G}(s)$, and $\psi$ be the length of a shortest 2-switch sequence transforming  $G$ into $H$. 
Then $\psi\leq |E(H)-E(G)|-1$.
\end{theorem} 

In \cite{BockRat} the authors study the matching number of trees with a given degree sequence, and of 
bipartite graphs with a given bipartite  degree sequence. The authors find minimum and maximum values 
for the matching number in these families, and then show that every value between the minimum and the maximum
is realized by a graph in the family. When this happens for a parameter, it is said
to have interval property with respect to the family of graphs being studied. To prove the interval property for bipartite degree sequences, they 
show that a 2-switch alters the matching number by at most $1$, and use the version of Theorem \ref{BergeTeo}
for bipartite degree sequences. In the case of trees, one could try to apply Theorem \ref{BergeTeo}
to go from a tree with the minimum possible matching number to a tree with the maximum possible matching number. Unfortunately
the intermediate graphs may not be trees (see Figure \ref{treeunitree}). Because of this the authors of \cite{BockRat}
 had to construct the tree realizing each
value between the minimum and the maximum.
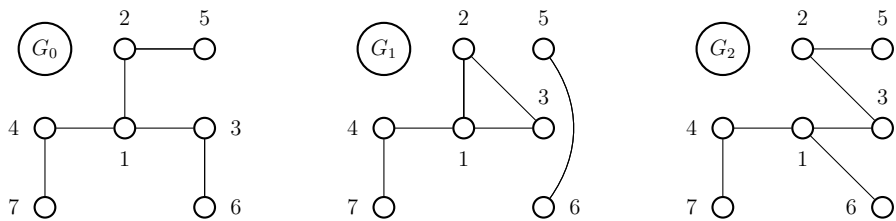
\begin{figure}[H]
\[
\begin{tikzpicture}
[scale=.7,auto=left,every node/.style={scale=.7,circle,thick,draw}] 
\node (0) at (-1.5,1.5) {$G_{0}$};
\node [label=below:1](1) at (0,0) {};
\node [label=left:4](4) at (-1.5,0) {};
\node [label=left:7](7) at (-1.5,-1.5) {};
\node [label=right:3](3) at (1.5,0) {};
\node [label=right:6](6) at (1.5,-1.5) {};
\node [label=above:2](2) at (0,1.5) {};
\node [label=above:5](5) at (1.5,1.5) {};
\draw (3) -- (1) ;
\draw (2) -- (1);
\draw (4) -- (1) ;
\draw (2) -- (5) ;
\draw (3) -- (6) ;
\foreach \from/\to in {2/5,3/6,4/7}
\draw (\from) -- (\to);
\end{tikzpicture}
\begin{tikzpicture}
\node (1) at (0,0) {};
\node (1) at (.75,0) {};
\end{tikzpicture}
\begin{tikzpicture}
[scale=.7,auto=left,every node/.style={scale=.7,circle,thick,draw}] 
\node (0) at (-1.5,1.5) {$G_{1}$};
\node [label=below:1](1) at (0,0) {};
\node [label=left:4](4) at (-1.5,0) {};
\node [label=left:7](7) at (-1.5,-1.5) {};
\node [label=3](3) at (1.5,0) {};
\node [label=right:6](6) at (1.5,-1.5) {};
\node [label=above:2](2) at (0,1.5) {};
\node [label=above:5](5) at (1.5,1.5) {};
\draw (3) -- (1) ;
\draw (2) -- (1);
\draw (4) -- (1) ;
\draw (3) -- (2);
\draw (6) [bend right=35] to (5) ;
\draw (1) -- (2) ;
\draw (5) [bend left=35] to (6) ;
\foreach \from/\to in {4/7,1/2}
\draw (\from) -- (\to);
\end{tikzpicture}
\begin{tikzpicture}
\node (1) at (0,0) {};
\node (1) at (.75,0) {};
\end{tikzpicture}
\begin{tikzpicture}
[scale=.7,auto=left,every node/.style={scale=.7,circle,thick,draw}] 
\node (0) at (-1.5,1.5) {$G_{2}$};
\node [label=below:1](1) at (0,0) {};
\node [label=left:4](4) at (-1.5,0) {};
\node [label=left:7](7) at (-1.5,-1.5) {};
\node [label=3](3) at (1.5,0) {};
\node [label=left:6](6) at (1.5,-1.5) {};
\node [label=above:2](2) at (0,1.5) {};
\node [label=above:5](5) at (1.5,1.5) {};
\draw (3) -- (1) ;
\draw (4) -- (1) ;
\draw (2) -- (5);
\draw (1) -- (6) ;
\foreach \from/\to in {4/7,2/3}
\draw (\from) -- (\to);
\end{tikzpicture}
\]
\caption{2-switch is not closed over trees.}
\label{treeunitree}
\end{figure}  

A natural question now arises: given two trees with the same  degree sequence, is there a way to obtain one from the other by applying successive 2-switches, in such a way that every intermediate graph is also a tree? 
Can this be done for any  other family of graphs?
In this paper we answer the first question in the positive for forests.
We also show that if two forests have the
same degree sequence, then they have the same number of connected components. This
implies also a positive answer to the first question for trees.
Concerning the second question, we show that there is no version of Theorem \ref{BergeTeo}
for bipartite graphs with a given degree sequence (not to be confused with bipartite degree sequence).
Afterwards, we apply the result to show that a plethora of parameters have the interval property with
respect to both $\mathcal{G}(s)$ and $\mathcal{F}(s)$.

The rest of the paper is presented as follows. In Section \ref{section2sasfunction} we introduce notation
to look at the 2-switch  operation as a function, and obtain some preliminary results. In Section \ref{sectiontswitch}
we characterize which 2-switches preserve the tree and the forest structure and we prove that if two forests have the
same degree sequence, then they have the same number of connected components. In Section \ref{sectionFTT}
we prove the analogous to Theorem \ref{BergeTeo} for forests with a given  degree sequence,
and present an Algorithm to find the necessary 2-switches that transform one forest into the other. In Section \ref{Stability_intevalProperty} we properly define the interval property, and find necessary conditions for a parameter
to satisfy the interval property in both $\mathcal{G}(s)$ and $\mathcal{F}(s)$. In Subsections \ref{sectionmatchingnumber},
\ref{sectionindependencenumber}, \ref{sectiondominationnumber}, \ref{sectionnumberofcomponents}, \ref{sectionpathcover},
\ref{sectionchromatic}, and \ref{sectionclique}, we show that the matching number, the independence number, the domination number, the number of components, the path-covering number, the chromatic number, and the clique number
have the interval property. Finally, in Section \ref{sectionconclusions} we give some final remarks
and present two non-isomorphic bipartite graphs with the same  degree sequence that require going through a nonbipartite graph
in order to transform one into the other through a sequence of 2-switches.

\section{2-Switch as a function}\label{section2sasfunction}

In order to define the 2-switch as a function, we need to introduce first the concept of interchangeability.


\begin{definition} \label{definterc}
    Let \(a,b,c,d \in [n]\) and let $G$ be a graph. The matrix ${{a \ b}\choose{c \ d}}$ is said to be \textbf{interchangeable} in $G$, if it satisfies the following conditions: 
	\begin{enumerate}
		\item $ab,cd\in E(G)$;
		\item $\{a,b\}\cap \{c,d\}=\varnothing$;
		\item $ac,bd\not\in E(G)$.
	\end{enumerate}
	Otherwise, ${{a \ b}\choose{c \ d}}$ is said to be \textbf{trivial} for $G$.	
\end{definition}
Notice in particular that if at least one of $a,b,c,d$ is not a vertex of $G$, then ${{a \ b}\choose{c \ d}}$ is trivial for $G$.

\begin{definition}
	\label{def2switch}
    Let \(n\) be a  integer and \(a,b,c,d \in [n]\), $A={{a \ b}\choose{c \ d}}$ and $G$ a graph. A \textbf{2-switch} is a function $\tau_{A}: \mathcal{G}\rightarrow \mathcal{G}$ defined as follows:
	\begin{equation}
    \tau_{A}(G)= \left\{ 
    \begin{array}{ll}
    G-ab-cd+ac+bd, &   \text{ if } A \text{ is interchangeable in } G, \\
    \\ G, & \text{ if } A \text{ is trivial for } G. \\
    \end{array}
    \right.
    \end{equation}

	If $\tau_{A}(G)=G$, we say that $\tau_{A}$ is \textbf{trivial} for $G$. The matrix $A$ is said to be an \textbf{action matrix} of $\tau_{A}$. 
\end{definition}

Clearly every 2-switch has associated an action matrix. Depending on the context, we can identify a 2-switch with an action matrix and omit it from the sub-index, i.e.: $\tau_{A}=\tau={{a \ b}\choose{c \ d}}$, and $\tau_{A}(G)=\tau(G)={{a \ b}\choose{c \ d}}G$. Notice that in general ${{a \ b}\choose{c \ d}} \neq {{a \ b}\choose{d \ c}}$, as a 2-switch, because the edges being added are different. It may even be the case that ${{a \ b}\choose{c \ d}}$ is interchangeable in $G$ while  ${{a \ b}\choose{d \ c}}$ it trivial for $G$. Moreover, in any action matrix of a 2-switch $\tau$, rows corresponds to those edges in $G$ that $\tau$ deletes, and columns corresponds to the edges of $G^{c}$ that $\tau$ adds to \(G\). Clearly, given any two disjoint edges $ab,cd\in G$, there are at most two different 2-switches in $G$ that act non-trivially on them. As expected, every 2-switch function preserves the degree sequence (i.e., $G\in \mathcal{G}(s) $ implies $\tau(G)\in \mathcal{G}(s)$).

From Definitions \ref{definterc} and \ref{def2switch} we can easily deduce the following.

\begin{lemma}
	\label{switchelemprop}
    Let \(n\) be a  integer and \(a,b,c,d \in [n]\). Let $A:={{a \ b}\choose{c \ d}}$ be a $2\times 2$ matrix and \(G \in \mathcal{G}_{n}\). Then:
	\begin{enumerate}
		\item if $P$ and $Q$ are $2\times 2$ permutation matrices, then $\tau_{A}=\tau_{PAQ}$;
		\item if $\tau$ is a nontrivial 2-switch in $G$, then $\left\|\tau(G)-G\right\|=\left\|G-\tau(G)\right\|=2$;
		\item if $ab$ and $cd$ are in distinct components of $G$, then ${{a \ b}\choose{c \ d}}$ and ${{a \ b}\choose{d \ c}}$ are interchangeable in $G$;
		\item if $A$ is interchangeable in $G$ and $A^t$ is the transpose of $A$, then $\tau_{A^{t}}$ is the unique 2-switch such that $\tau_{A^{t}}\tau_{A}(G)=G$ (and hence $\tau_{A^{t}}$ is the inverse 2-switch of $\tau_{A}$). 
	\end{enumerate}
\end{lemma}

\begin{proof}
    The proofs of these statements follow directly from the definition of $2$-switch and are left to the reader.
\end{proof}

One of the advantages of defining 2-switch as a function is that we can describe 2-switch sequences in terms of compositions. Thus, saying that $(\tau_{i})=(\tau_{i})_{1\leq i\leq r}$ is a 2-switch sequence transforming $G$ into $H$ is equivalent to writing $H=\tau_{r}\ldots \tau_{1}(G)$, where we write composition as product. We say that the sequence of 2-switches \((\tau_{i})_{i=1}^{r}\) has \textit{length} \(r\). We assume that the empty sequence $(\varnothing)$ transforms every graph into itself and has length 0.


\section{T-Switch and F-Switch}\label{sectiontswitch}

In order to obtain a version of Theorem \ref{BergeTeo} for the family $\mathcal{F}(s)$, we need to characterize those 2-switches over a forests that preserve the forest structure. But first we do it for trees.

\begin{definition}\label{tsdef}
	A nontrivial 2-switch $\tau$ over a tree $T$ is said to be a \textbf{t-switch} if $\tau(T)$ is a tree.
\end{definition}

\begin{theorem}\label{tsiff}
    Let $T$ be a tree and $ab,cd$ two non-incident edges in $T$. A 2-switch $\tau={{a \ b}\choose{c \ d}}$ is a t-switch if and only if the path in \(T\) between \(a\) and \(d\) has the form $(ab\ldots cd)$ or  the path in \(T\) between \(b\) and \(c\) has the form $(ba\ldots dc)$. 
\end{theorem}

\begin{proof}
For the only if part, suppose that the unique path in \(T\) between \(a\) and \(d\) has the form $(ab\ldots cd)$. By Definition \ref{tsdef}, we need to verify that $\tau(T)$ is a tree. In order to do this,  divide the action of $\tau$ in three parts. First, notice that ${{a \ b}\choose{c \ d}}$ is interchangeable in $T$. After erasing $ab$ and $cd$, $T$ splits into a forest with three connected components $T_{1}$, $T_2$ and $T_3$, such that $a\in T_{1}$, $(b\ldots c)\subset T_{2}$ and $d\in T_{3}$. Then, we connect $T_{1}$ to $T_{2}$ with $ac$, obtaining a new forest with two components, $T_{3}$ and $T_{12}$, where $T_{12}$ contains the path $(b\ldots ca)$. Finally, we get $\tau(T)$ by connecting $T_{3}$ to $T_{12}$ through $bd$. Hence, $\tau(T)$ is a tree. 

If the unique path in \(T\) between \(b\) and \(c\) has the form $(ba\ldots dc)$, since
\[
{{a \ b}\choose{c \ d}} = {{1 \ 0}\choose{0 \ 1}}{{b \ a}\choose{d \ c}} {{0 \ 1}\choose{1 \ 0}},
\]
by the previous case and Lemma \ref{switchelemprop} we conclude that also $\tau(T)$ is a tree in this case.
Therefore, if the path in \(T\) between \(a\) and \(d\) has the form $(ab\ldots cd)$ or  the path in \(T\) between \(b\) and \(c\) has the form $(ba\ldots dc)$, then $\tau={{a \ b}\choose{c \ d}}$ is a t-switch.

Conversely, suppose that neither the unique path between \(a \) and \(d\) in \(T\) has the form $(ab\ldots cd)$ nor the unique path between \(b \) and \(c\) in \(T\) has the form $(ba\ldots dc)$. Then, either the path between $a$ and $c$
has the form $(ab\ldots dc)$ or the path between $b$ and $d$ has the form $(ba\ldots cd)$. In the first case, if $\tau$ is not trivial, $\tau(T)$ contains $2$ paths between $b$ and $d$. Similarly, in the second case, if $\tau$ is not trivial, $\tau(T)$ contains $2$ paths between $a$ and $c$. In either case, $\tau$ is not a t-switch.
Therefore, if $\tau={{a \ b}\choose{c \ d}}$ is a t-switch then the path in \(T\) between \(a\) and \(d\) has the form $(ab\ldots cd)$ or  the path in \(T\) between \(b\) and \(c\) has the form $(ba\ldots dc)$. 
\end{proof}

The next easy to prove fact will be used several times along this work.  

\begin{theorem}
	\label{=s=k}
	Any two forests with the same degree sequence have the same number of connected components.
\end{theorem}
\begin{proof}
Let $s=(d_{1},\ldots ,d_{n})$ be the  degree sequence of the forests $F_{1}$ and $F_{2}$. Then, \[n-\kappa(F_{1})=\left\|F_{1}\right\|=\frac{1}{2} \sum_{i=1}^{n}d_{i}=\left\|F_{2}\right\|=n-\kappa(F_{2}). \qedhere\]
\end{proof}

If $\mathcal{F}(s)$ contains a tree, then Theorem \ref{=s=k} implies, in particular, that every graph in $\mathcal{F}(s)$ is a tree. This means that any result obtained for forests with a given degree sequence also holds for trees with a given degree sequence. A useful way to look at Theorem \ref{=s=k} is that there is no sequence of 2-switches transforming a forest into another with a different number of components. 

\begin{definition}
    \label{fsdef}
    A nontrivial 2-switch $\tau$ over a forest $F$ is said to be an \textbf{f-switch} if $\tau(F)$ is a forest.
\end{definition}

Using t-switch, we can characterize when a 2-switch over a forest is an f-switch. It will be useful and intuitive from now on to say that $\tau={{a \ b}\choose{c \ d}}$ is a 2-switch \textit{between} the edges $ab$ and $cd$ of $G$, whenever $\tau$ is nontrivial in $G$.

\begin{theorem}\label{teofswitch}
	Let $\tau$ be a 2-switch between two disjoint edges $e_{1}$ and $e_{2}$ of a forest $F$. Let $\tau$ be a 2-switch over \(F\) between $e_{1}$ and $e_{2}$. The 2-switch \(\tau\) is an f-switch over $F$ if and only if: 
	\begin{enumerate}
		\item $\tau$ is a t-switch, if $e_{1}$ and $e_{2}$ are in the same connected component; 
		\item $\tau$ is a 2-switch, if $e_{1}$ and $e_{2}$ are in different connected components.
	\end{enumerate}
\end{theorem}

\begin{proof}
Let \(e'_{1}\) and \(e'_{2}\) be the edges that \(\tau\) adds to \(F\). For the only if part, assume that $\tau$ acts in that way over $F$. By Definition \ref{fsdef}, we need to verify that $\tau(F)$ is a forest. There is nothing to prove if $e_{1}$ and $e_{2}$ are in the same component, because we already know that t-switches preserve the tree structure. If $e_{1}\in T_{1}$ and $e_{2}\in T_{2}$, where $T_{1}$ and $T_{2}$ are two different components of $F$, we analyze the action of $\tau$ step by step. First, by erasing $e_{j}$, the component $T_{j}$ splits into the  sub-components $T'_{j}$ and $T''_{j}$. The number of components of the graph increases momentarily by 2 and each of the vertices involved in $\tau$ lies now in a distinct sub-component. Then, by adding the new edges $e'_{1}$ and $e'_{2}$, we see that $e'_{1}$ connects $T'_{1}$ to $T'_{2}$ (or $T'_{1}$ to $T''_{2}$) and $e'_{2}$ connects $T''_{1}$ to $T''_{2}$ (or $T''_{1}$ to $T'_{2}$). So, $\tau(F)$ is a forest. 

For the if part, simply note that if $\tau$ was not a t-switch, then $\tau(F)$ would contain a cycle. 
\end{proof}

\section{The forest transition theorem}\label{sectionFTT}
In this section we are going to prove that given two forests $F_1$ and $F_2$ with the same degree sequence, there is a sequences of $f$-switches that transforms $F_1$ into $F_2$.
The process works by deleting vertices of degree one whose incident edge is in $E(F_1)\cap E(F_2)$,
and using $f$-switches to obtain more of those vertices once we cannot delete any more.
Vertices of degree \(1\) are called leaves.  Notice that if $W$ is a set of leaves of $G$, then \(\kappa(G)\geq\kappa(G- W)\). Furthermore, \(\kappa(G)>\kappa(G- W)\) only if we delete leaves that are neighbors of each other. Let \(s\) be a graphical sequence and \(G, H \in \mathcal{G}(s)\). Notice that $s$ determines which vertices are leaves,
hence $G$ and $H$ have the same set of leaves. A leaf $\ell$ is said to be a \textbf{trimmable leaf} of the graphs $G$ and $H$ if $\ell$ has the same neighbor in both $G$ and $H$.  We denote the set of trimmable leaves between \(G\) and \(H\) by \(\Lambda(G,H)\), or just \(\Lambda\) when \(G\) and \(H\) are clear from the context. 
As an example, the set of trimmable leaves between the graphs $G_0$ and $G_2$ in Figure \ref{treeunitree}
are $\Lambda(G_0,G_2)=\{5,7\}$ because $52$ and $74$ are in both $G_0$ and $G_2$, whereas the leaf $6$ is not trimmable because it is adjacent to $3$ in $G_0$ and to $1$ in $G_2$.

The next lemma is one of the preliminary steps to the main result of this section, which will use the simple idea of consecutive deletions of trimmable leaves to transform a forest into another with the same  degree sequence.

\begin{lemma}
	\label{tauiL}
	Let $s$ be a graphical  sequence and $F,F'\in \mathcal{F}(s)$. Let $\Lambda$ be a set of trimmable leaves of $F$ and $F'$. Suppose that $(\tau_{i})_{i=1}^{r}$ is an f-switch sequence transforming $F-\Lambda$ into $F'-\Lambda$. Then, $(\tau_{i})_{i=1}^{r}$ is an f-switch sequence transforming $F$ into $F'$.
\end{lemma}
\begin{proof}
Let $F_0=F$ and $F_i=\tau_i(F_{i-1})$. As none of the vertices in $\Lambda$ is involved in any of the f-switches
(recall that trivial $2$-switches are not f-switches), $F_i-\Lambda=\tau_i(F_{i-1}-\Lambda)$ for every $1\leq i\leq r$.
Hence $(\tau_i)$ is a sequence of $2$-switches transforming $F$ into $F'$.

We need to see now that $(\tau_i)$ are f-switches for the transformation of $F$ into $F'$.
But, as $F_i-\Lambda$ is obtained from $F_i$ by removing vertices of degree $1$, hence $F_i-\Lambda$ has the same
cycles as $F_i$. Thus, as every $F_i-\Lambda$ is a forest, every $F_i$ is a forest.
Therefore $(\tau_i)$ is a sequence of f-switches transforming $F$ into $F'$.
\end{proof}

\begin{lemma} \label{ponerhojascompartidas}
	Let $s$ be a graphical sequence and $F,F'\in \mathcal{F}(s)$. If \(\Lambda(F,F')=\varnothing\), then there exists an f-switch \(\tau\) over \(F\) such that  \(\Lambda(\tau(F),F')\neq \varnothing\).
\end{lemma}

\begin{proof}
Assume $\Lambda(F,F')=\varnothing$. 
We split the proof in two cases: either there is a leaf whose neighbor in $F'$  has degree at least $2$, or every vertex has degree $1$.

For the first case, let $\ell$ be a leaf such that its neighbor $u$ in $F'$ has degree at least $2$,  and let $v$ be the neighbor of $\ell$ in $F$.
If $\ell$ and $u$ are in different connected components of $F$, then let $w$ be a neighbor of $u$ in 
$F$ and perform the $2$-switch $\tau={{\ell \ v}\choose{u \ w}}$. Notice that $\tau$ is an f-switch because 
$v\ell$ and $uw$ are in different connected components of $F$.
If $\ell$ and $u$ are in the same connected component, let $(\ell v \ldots u)$ be the path from $\ell$ to $u$ in $F$.
As $\deg u \geq 2$, there is a neighbor of $u$ that is not in  $(\ell v \ldots u)$. Let $w$ be such a neighbor.
Then the $2$-switch $\tau={{\ell \ v}\choose{u \ w}}$ is an f-switch by Theorems \ref{tsiff} and \ref{teofswitch}.
In either case, $u\ell \in E(\tau(F))$ and $u\in \Lambda(\tau(F),F')$.

For the second case, let $\ell$ be any leaf, let $v$ and $u$ be the neighbors of $\ell$ in $F$ and $F'$ respectively,
and let $w$ be the neighbor of $u$ in $F$.
The $2$-switch $\tau={{\ell \ v}\choose{u \ w}}$ is an f-switch, because $v\ell$ and $uw$ are in different connected components. Furthermore, $u\in \Lambda(\tau(F),F')$.

Therefore, there is an f-switch   \(\tau\) over \(F\) such that  \(\Lambda(\tau(F),F')\neq \varnothing\).
\end{proof}

The next result states that given two forests with the same degree sequence, there is a sequence of f-switches transforming one into the other.
Before proceeding with the proof, we need to note two things. First, it is sufficient to prove the result for forests without
isolated vertices, because any two forests with the same degree sequence coincide in their isolated vertices trivially. It is easy to check by inspection that the result follows for forests of order $n\leq 4$.
We are now ready to proceed.

\begin{theorem}
	\label{IPTT}
	Let $s$ be a graphical sequence and $F,F'\in \mathcal{F}(s)$. Then there is a sequence of f-switches transforming $F$ into $F'$.
\end{theorem}
\begin{proof}
Suppose $F$ and $F'$ have no isolated vertices.
We use induction on the order $n$ of $F$ and $F'$. If $n\leq 4$, the statement is true. Hence, let $n>4$, and suppose that every pair of forests of order less than $n$ with the same degree sequence can be transformed into each other by a sequence of f-switches. We have two cases: $\Lambda(F,F')\neq \varnothing$ and $\Lambda(F,F')= \varnothing$.

If $\Lambda(F,F')\neq \varnothing$, consider $F-\Lambda$ and $F'-\Lambda$. These are two forests of order $n_{1}<n$, with the same  degree sequence $s_{1}$. So, inductive hypothesis applies to $F-\Lambda$ and $F'-\Lambda$: there exists an f-switch sequence $(\tau_{i})$ transforming $F-\Lambda$ into $F'-\Lambda$. Hence, by Lemma \ref{tauiL}, the sequence $(\tau_{i})$ is an f-switch sequence from $F$ to $F'$.

The case $\Lambda(F,F') = \varnothing$, by Lemma \ref{ponerhojascompartidas}, can be reduced to the previous case. \qedhere
\end{proof}

Theorem \ref{IPTT} guarantees not only the existence of a transforming f-switch sequence between any two forests in $\mathcal{F}(s)$; moreover, its proof contains an algorithm that returns a transforming f-switch sequence between any two forest of \(\mathcal{F}(s)\). We make explicit such an algorithm.

%
\begin{algorithm}
	\label{algtrans}
	\textbf{Transition algorithm}\\
	INPUT: two forests $F,F' \in \mathcal{F}(s)$.
	 \begin{enumerate}
	 	\item If \(F=F'\): RETURN \((\varnothing)\).
		\item Let $r=0$ and $\Lambda=\Lambda(F,F')$.
		\item While \(F\neq F'\):
			\begin{enumerate}
			\item If $\Lambda=\varnothing$:
				\begin{enumerate}
					\item Let \(r=r+1\).
					\item If every vertex in $F'$ has degree $1$ choose a leaf \(\ell \in V(F')\). Else choose a leaf 
					$\ell \in V(F')$ such that its neighbor \(u\) in \(F'\) has degree at least $2$.		
					\item Find the f-switch \(\tau\) such that \(\ell\) is trimmable between \(\tau(F)\) and \(F'\) (such f-switch exists by Lemma \ref{ponerhojascompartidas}).
					\item Let \(\tau_{r}=\tau\), \(F=\tau(F)\), and $\Lambda=\Lambda(F,F')$.
				\end{enumerate}
			\item Let \(F=F-\Lambda\) and \(F'=F'-\Lambda\).
			\end{enumerate}
		\item RETURN \((\tau_{i})_{i=1}^{r}\).
	\end{enumerate}	
\end{algorithm}

%
Note that in each step of Transition Algorithm either \(F=F'\) or at least one leaf is removed (equivalently, one edge is removed). Thus, the Transition Algorithm runs at most \(|E(F')-E(F)|\) times.
This number can be improved by $1$.
First we need a technical lemma.

\begin{lemma}
	\label{etaneq1}
	Let \(s\) be a graphical sequence and let \(G,H \in \mathcal{G}(s)\). Then \(| E(G) - E(H)|\neq 1\). 
\end{lemma}
\begin{proof}
We proceed by contradiction.
	Assume that \(| E(G) - E(H)|= 1\). Notice that $G$ and $H$ having the same degree sequence implies that $|E(G)|=|E(H)|$. Thus, $| E(H) - E(G)|= 1$.

Let $ab$ be the only edge in $E(G)-E(H)$.	
	As $G$ and $H$ have the same degree  sequence, $\deg_G(a)=\deg_H a$ and $\deg_G(b)=\deg_H(b)$. Hence there must be an edge incident to $a$ and 
	an edge incident to $b$ in $E(H)-E(G)$. 
But as 	$| E(H) - E(G)|= 1$, this can only be possible if $ab$ is an edge of $H$, contradicting the fact that
$ab$ is the only edge in $E(G)-E(H)$.
\end{proof}

Suppose now that $(\tau_{i})_{i=1^{r}}$ is the f-switch sequence obtained as output of Transition Algorithm applied to two forests \(F\) and \(F'\). Let $F_0=F$ and $F_i=\tau_i(F_{i-1})$ for $1\leq i \leq r$.
Notice that for each $i\in [r]$,  $|E(F')\cap E(F_i)|\geq |E(F')\cap E(F_{i-1})|+1$, thus
\[
\sum_{i=1}^{r-1}\left( |E(F')\cap E(F_i)|-|E(F')\cap E(F_{i-1})| \right) \geq r.
\]
Since this last sum is telescoping we have  
\[
|E(F')\cap E(F_{r-1})|-|E(F')\cap E(F)|\geq r-1.
\]
On the other hand, Lemma \ref{etaneq1} implies
\begin{align*}
|E(F')\cap E(F_{r})|-|E(F')\cap E(F_{r-1})|=|E(F')|-|E(F')\cap E(F_{r-1})|\geq 2.
\end{align*}
Thus,
\begin{align*}
|E(F')\cap E(F_{r})|-|E(F')\cap E(F)|\geq r+1.
\end{align*}
But $E(F_{r})=E(F')$, and $|E(F')|-|E(F')\cap E(F)|=|E(F')-E(F)|$.
Therefore 
\begin{align*}
r\leq |E(F')-E(F)|-1.
\end{align*}
Hence the Transition Algorithm runs at most $|E(F')-E(F)|-1$ times.

The previous discussion together with Theorem \ref{IPTT} give us the main result of the section: 
\begin{theorem}[Forest Transition Theorem]
	\label{transbosques}
		Let $s$ be a graphical sequence and let $F$ and $F'$ two forests in $\mathcal{F}(s)$. Then \(F\) can be transformed into \(F'\) with at most $|E(F')-E(F)|-1$ f-switches given by the Transition Algorithm.
\end{theorem}

It is important to note that even though the bound in Theorem \ref{transbosques} is a priori worse than the result in \cite{BeregIto}, the sequence of 2-switches that
they construct is not necessarily a sequence of f-switches. 
\section{Stability and interval property}\label{Stability_intevalProperty}\label{sectionIntervalproperty}

In this section we apply the transition theorems to study how the 2-switch operations perturbs different parameters on graphs.
We introduce the notion of stability.

\begin{definition}
	A graph parameter \(\xi\) is said to be \textbf{stable} under $2$-switch, if given \(G\) a graph and \(\tau\) a $2$-switch, then
	\[
	\left| \xi\left(\tau(G)\right)-\xi(G)\right| \leq 1.
	\]
\end{definition}

The next result provides an easy way to show that a parameter is stable under $2$-switch.
\begin{lemma}\label{lemastable}
Let $\xi$ be an integer parameter. The following hold
\begin{enumerate}
\item if $\xi(\tau(G))\leq \xi(G)+1$ for every graph $G$ and every $2$-switch $\tau$, then $\xi$ is stable under $2$-switch; 
\item if $\xi(\tau(G))\geq \xi(G)-1$ for every graph $G$ and every $2$-switch $\tau$, then $\xi$ is stable under $2$-switch.
\end{enumerate}
\end{lemma}
\begin{proof}
In order to prove the first implication, assume \begin{equation}\label{eqlemastable}
\xi(\tau(G))\leq \xi(G)+1
\end{equation} 
for every graph $G$ and every $2$-switch $\tau$.

Fix $G$  and $\tau$. We have then $\xi(\tau(G))-\tau(G)\leq 1$.
On the other hand by Lemma \ref{switchelemprop}, $\tau$ has 
an inverse $2$-switch $\tau^{-1}$. 
Applying Inequality (\ref{eqlemastable}) to $\tau^{-1}$ and $\tau(G)$ yields
\[
\xi(G)=\xi(\tau^{-1}(\tau(G)))\leq \xi(\tau(G))+1.
\]
Hence,
\[
-1\leq \xi(\tau(G))-\xi(G).
\]
Therefore,
\[
\left| \xi\left(\tau(G)\right)-\xi(G)\right| \leq 1,
\] 
and $\xi$ is stable.

The proof of the second implication is similar and is left to the reader.
\end{proof}

As claimed in the introduction, we are interested in finding which values a parameter $\xi$ realizes in $\mathcal{F}(s)$ and $\mathcal{G}(s)$ . In order to do that, we need to introduce the concept of interval property. One could think of this property as a discrete analogous of the Intermediate Value Theorem from elementary Calculus.
\begin{definition}
	Let $\mathcal{X}$ be a collection of graphs and let $\xi:\mathcal{X}\longrightarrow \mathbb{R}$ be a parameter defined on $\mathcal{X}$. We say that $\xi$ has the \textbf{interval property} on $\mathcal{X}$ if $\xi(\mathcal{X})=I\cap \mathbb{Z}$, for some interval $I\subset \mathbb{R}$. 
\end{definition}
In other words, if $\xi_{\max}$ and $\xi_{\min}$ are the maximum and minimum values obtain by $\xi(G)$ with $G\in \mathcal{X}$, then $\xi$ has the interval property on $\mathcal{X}$ if for every integer $k$ such that $\xi_{\min}\leq k\leq \xi_{\max}$ there is a $G\in\mathcal{X}$) such that $\xi(G)=k$.

The next theorem shows that stability implies interval property for integer parameters. This idea was used in \cite{BockRat}
to prove the interval property for the matching number on the family of bipartite graphs with a given bipartite degree sequence, but here we formalize it in order to apply it to many parameters.

\begin{theorem} \label{printparamgeneral}
	Let \(s\) be a graphical sequence and \(\xi\)  an integer parameter. If \(\xi\) is stable under 2-switch, then \(\xi\) has the interval property on \(\mathcal{G}(s)\) and on \(\mathcal{F}(s)\) . 
\end{theorem}
\begin{proof}
Let $\mathcal{X}\in\{\mathcal{G}(s),\mathcal{F}(s)\}$.
Consider two graphs $G_{1},G_{2}\in \mathcal{X}$ such that $\xi(G_{1})$ and $\xi(G_{2})$ are respectively the minimum and the maximum value for $\xi$ on $\mathcal{X}$. If $\mathcal{X}=\mathcal{G}(s)$, by Theorem \ref{BergeTeo}, there exists a 2-switch sequence $(\tau_{i})$ transforming $G_{1}$ into $G_{2}$. 
Otherwise, if $\mathcal{X}=\mathcal{F}(s)$, by Theorem \ref{transbosques} there is an f-switch sequence transforming $G_{1}$ into $G_{2}$.
In either case, since each $\tau_{i}$ perturbs $\xi$ at most by 1, every integer value in the interval $[\xi(G_{1}),\xi(G_{2})]$ must be attained in some graph of the transition. This means that $\xi$ has the interval property on $\mathcal{X}$.
\end{proof}

Note that Theorem \ref{printparamgeneral} gives us a way to obtain a forest with any (allowed) value of \(\xi\).  Assume that $\xi$ is stable under 2-switch. Given two forests $F_{1},F_{2}\in \mathcal{F}(s)$ such that $\xi_{1}:=\xi(F_{1})$ and $\xi_{2}:=\xi(F_{2})$ and $\xi_{1}\leq\xi_{2}$. Fix an integer $k\in [\xi_{1},\xi_{2}]$. Then, apply the Transition Algorithm from $F_{1}$ to $F_{2}$ to obtain a forest $F$ such that $\xi(F)=k$.

We use Theorem \ref{printparamgeneral} to study the stability and the interval property of several integer  parameters.

\subsection{Matching number and related parameters}\label{sectionmatchingnumber}

A matching in a graph $G$ is a set of pairwise disjoint edges of $G$. The maximum cardinality of a matching in $G$ is called the matching number of $G$, which is denoted by $\mu(G)$. A matching in \(G\) with maximum cardinality is called a maximum matching. A proof of the stability of $\mu$ under 2-switch can be found between the lines of \cite{BockRat}.
We include our proof for completion.
\begin{lemma}
	\label{matchlem1}
	Let $M$ be a maximum matching in a graph $G$ and let $\tau$ be a 2-switch between $e_{1},e_{2}\in E(G)$. If $e_{1}$ and $e_{2}$ are both in $M$ or both in $G-M$, then $\mu(G)\leq \mu(\tau (G))$.
\end{lemma}

\begin{proof}
If $e_{1},e_{2}\in E(G)-M$, then $M$ is also a matching in $\tau (G)$.  Hence, $|M|=\mu(G)\leq \mu(\tau (G))$. 

If $e_{1},e_{2}\in M$, the set $M'=M-\{e_{1},e_{2}\}$ is a matching of size $\mu(G)-2$ in $\tau (G)=(G-\{e_{1},e_{2}\})\cup \{e'_{1},e'_{2}\}$, where \(e'_{1}\) and \(e'_{2}\) are the edges in \(G^{c}\) that \(\tau\)  add to \(G\). Notice that none of the four vertices involved in $\tau$ belong to some edge of $M'$. Hence, \(M'\cup \{e'_{1},e'_{2}\}\) is a matching of \(\tau (G)\).  Therefore, $\mu(G')\geq |M'\cup \{e'_{1},e'_{2}\}|=(\mu(G)-2)+2=\mu(G)$.
\end{proof}

\begin{lemma}
	\label{matchlem2}
	 Let $M$ be a maximum matching in a graph $G$, and $\tau$ be a 2-switch between $e_{1},e_{2}\in G$. If $e_{1}\in M$ and $e_{2}\notin M$, then $\mu(\tau (G))\geq \mu(G)-1$.
\end{lemma}
\begin{proof}
The set $M-e_{1}$ is a matching in $\tau(G)$ of size $\mu(G)-1$. Thus, $\mu(\tau(G))\geq \mu(G)-1$. 
\end{proof}

Now we can apply Lemmas \ref{lemastable}, \ref{matchlem1} and \ref{matchlem2} to obtain the following.
\begin{theorem}
	\label{matchlem3}
	The matching number is stable under 2-switch.
\end{theorem}  

\begin{proof}
Let $G$ be a graph and $\tau$ be a 2-switch. By Lemmas \ref{matchlem1} and \ref{matchlem2}, 
$\mu(\tau(G))\geq \mu(G)-1$. Thus, by Lemma \ref{lemastable} $\mu$ is stable under $2$-switch.
\end{proof}

Theorem \ref{matchlem3} together with Theorem \ref{printparamgeneral} yields the following result. 

\begin{corollary}
	 Let \(s\) be a graphical sequence. The matching number has the interval property on $\mathcal{G}(s)$ and on $\mathcal{F}(s)$.
\end{corollary}

Let $G$ be a graph. An edge cover of $G$ is a set of edges $\mathcal{C}$ such that every vertex of $G$ is incident to at least one edge of $\mathcal{C}$. A minimum edge cover of $G$ is an edge cover of $G$ of minimum cardinality. The edge-covering number of $G$, denoted by $\epsilon(G)$, is the cardinality of a minimum edge cover of $G$. It is known that $\epsilon(G)=n-\mu(G)$, where $n$ is the order of $G$ (see \cite{Gallai}). The next corollaries follow easily.

\begin{corollary}
	\label{edgecoverprint}
	The edge-covering number is stable under 2-switch.
\end{corollary}
\begin{corollary}
	Let \(s\) be a graphical sequence. The edge-covering number has the interval property on $\mathcal{G}(s)$ and on $\mathcal{F}(s)$.
\end{corollary}

The rank and nullity of a graph are the rank and nullity of its adjacency matrix. It is known that $\rank(F)=2\mu(F)$, for any forest $F$ (see \cite{Bevis,Gutman}). Combining this fact with the Rank-nullity theorem of linear algebra and the interval property of the matching number on $\mathcal{F}(s)$, we get the following results.

\begin{corollary}\label{cororanknull}
	Let \(s\) be a graphical sequence. Let $F\in \mathcal{F}(s)$ be a forest and \(\tau\) an f-switch over \(F\). Then \[|\rank(\tau(F))-\rank(F)|=|\mnull(\tau(F))-\mnull(F)|\in \{0,2\}.\]
\end{corollary}
Notice that Corollary \ref{cororanknull} implies that a property similar to the interval property for $\rank$ and $\mnull$, except that $\rank$ takes only even values and $\mnull$ either only even or only odd. Hence we get the following.
\begin{corollary}
	Let \(s\) be a graphical sequence of length \(n\). Then, there exists an interval \(I\subset \mathbb{R}\) such that   
	\begin{enumerate}
		\item $\rank(\mathcal{F}(s))=I\cap 2\mathbb{Z}$,
		\item $\mnull(\mathcal{F}(s))=(n-I)\cap (2\mathbb{Z}+n)$.
	\end{enumerate}
\end{corollary}


\subsection{Independence number and vertex-covering number}\label{sectionindependencenumber}

An independent set of a graph $G$ is a set of vertices in $G$, no two of which are adjacent. A maximum independent set in $G$ is an independent set of \(G\) with the largest possible cardinality. This cardinality is called the independence number of $G$, and it is denoted by $\alpha(G)$.

\begin{theorem}\label{teoindepstable}
	The independence number is stable under 2-switch.
\end{theorem}  
\begin{proof}
Let \(G\) be a graph of order \(n\) with vertex set \([n]\), $I\subset [n]$ be a maximum independent set in $G$ and $\tau={{a \ b}\choose{c \ d}}$ be a 2-switch over \(G\). 
Notice that $|I\cap \{a,b,c,d\}|\leq 2$.

If \(\left| I \cap  \{a,b,c,d\} \right| \leq 1\), then \(\alpha(\tau(G)) \geq |I|=\alpha(G)\),  because $I$ is an independent set in $\tau(G)$. We can easily conclude the same when $I\cap \{a,b,c,d\}=\{a,d\}$ or $I\cap \{a,b,c,d\}=\{b,c\}$. 

If $I\cap \{a,b,c,d\}=\{a,c\}$, notice that $I$ is not an independent set in $\tau(G)$, since  $a$ is adjacent to $c$ in $\tau(G)$. Thus, $I-a$ is an independent set in $\tau(G)$. Hence, $\alpha(\tau(G))\geq |I-a|=\alpha(G)-1$. The same argument holds if $I\cap \{a,b,c,d\}=\{b,d\}$. 

In either case $\alpha(\tau(G))\geq \alpha(G)-1$. Thus, by Lemma \ref{lemastable} $\alpha$ is stable under 2-switch.  
\end{proof}

Theorem \ref{teoindepstable} together with Theorem \ref{printparamgeneral} imply the following.

\begin{corollary}
Let $s$ be a graphical sequence. The independence number has the interval property on $\mathcal{G}(s)$ and on $\mathcal{F}(s)$.
\end{corollary}

A vertex cover of a graph $G$ is a set of vertices $U\subset V(G)$ such that each edge of $G$ is incident to at least one vertex of the set $U$. A minimum vertex cover of $G$ is a vertex cover of $G$ of minimum cardinality. The vertex-covering number of $G$, denoted by $\nu(G)$, is the cardinality of a minimum vertex cover of $G$. It is known that $\nu(G)=n-\alpha(G)$, where $n$ is the order of $G$. Therefore, the results for the independence number imply similar results for the vertex-covering number.

\begin{corollary}
	The vertex-covering number is stable under 2-switch.
\end{corollary}
\begin{corollary}
	Let \(s\) be a graphical sequence. The vertex-covering number has the interval property on $\mathcal{G}(s)$ and on $\mathcal{F}(s)$.
\end{corollary}


\subsection{Domination number}\label{sectiondominationnumber}

A dominating set of a graph $G$ is a set $D$ of vertices such that every vertex of $G$ not in $D$ is adjacent to at least one element of $D$. Under this condition, we say that $D$ dominates (or covers) a vertex $v$ if $v$ is adjacent to some vertex of $D$ or if $v\in D$. A minimum dominating set is a dominating set of minimum cardinality. The domination number of $G$, denoted by $\gamma(G)$, is the cardinality of a minimum dominating set of $G$.

\begin{lemma}
	\label{lemdom}
	The domination number is stable under $2$-switch.
\end{lemma}

\begin{proof}
	Let $G$ be a graph of order $n\geq 4$, $D$ a minimum dominating set on $G$, and $\tau={{a \ b}\choose{c \ d}}$ a 2-switch over \(G\). 
	
	If $D$ is a dominating set in $\tau(G)$, then $\gamma(\tau(G))\leq |D|=\gamma(G)\leq \gamma(G)+1$. 
	
	Assume $D$ is not a dominating set in $\tau(G)$. As the edges incident to vertices not in $\{a,b,c,d\}$ in $G$ and in $\tau(G)$ are the same, $D$ dominates every vertex in $V(G)-\{a,b,c,d\}$. Hence at least one vertex in $\{a,b,c,d\}$ is not dominated by \(D\) in $\tau(G)$. Without loss of generality, assume that vertex is $a$ and consider its neighbors in \(G\) and \(\tau(G)\). Since the only edge incident to $a$ in $G$ that is not in $\tau(G)$ is $ab$, $b$ must be in $D$. Therefore, $d$ is dominated in $\tau(G)$ by $b$. Moreover, $D\cup a$ is a dominating set in $\tau(G)$, because $c$ is dominated by $a$ in $\tau(G)$. Thus, $\gamma(\tau(G))\leq |D\cup a|=\gamma(G)+1$. 
	
	As in either case $\gamma(G)\leq \gamma(G)+1$, Lemma \ref{lemastable} implies that $\gamma$ is stable under 2-switch.
\end{proof}

Lemma \ref{lemdom} together with Theorem \ref{printparamgeneral} yields the following. 

\begin{corollary}
	Let \(s\) be a graphical sequence. The domination number has the interval property on $\mathcal{G}(s)$ and on $\mathcal{F}(s)$.
\end{corollary}

\subsection{Number of connected components}\label{sectionnumberofcomponents}

Previously, with Theorem \ref{=s=k}, we established that two forests with the same degree sequence $s$ must have the same number of connected components. In this sense we can rephrase Theorem \ref{=s=k} by saying that $\kappa$ is constant on $\mathcal{F}(s)$, and therefore it trivially has the interval property on $\mathcal{F}(s)$. We will prove that the same occurs on $\mathcal{G}(s)$. 

\begin{theorem}\label{teokappa}
The number of connected components is stable under 2-switch.
\end{theorem}
\begin{proof}
Let \(G\) be a graph, \(\tau={{a \ b}\choose{c \ d}}\) a 2-switch over \(G\), and let $G'=G-ab-cd$.

Clearly $\kappa(G')\leq \kappa(G)+2$, as deleting an edge increases the number of connected components in at most 1.
If $\kappa(G')\leq \kappa(G)+1$, then $\kappa(\tau(G))=\kappa(G'+ac+bd)\leq \kappa(G)+1$, as adding edges
cannot increase the number of connected components.

Assume $\kappa(G')=\kappa(G)+2$. 
If $a$ and $c$ are in different connected components of $G'$, then $\kappa(G'+ac)=\kappa(G')-1\leq \kappa(G)+1$,
which implies $\kappa(\tau(G))\leq \kappa(G)+1$.
Suppose $a$ and $c$ are in the same connected component of $G'$
and denote such component by $C_{ac}$. Then neither $b$ nor $d$ are in $C_{ac}$, because $\kappa(G')=\kappa(G)+2$.
Furthermore, $b$ and $d$ must be in different connected components, otherwise $\kappa(G)=\kappa(G'+ab+cd)=\kappa(G')-1$. 
Hence, $\kappa(G'+bd)=\kappa(G')-1\leq \kappa(G)+1$, implying that $\kappa(\tau(G))\leq \kappa(G)+1$.

Therefore, Lemma \ref{lemastable} implies that $\kappa$ is stable under 2-switch.
\end{proof}

Theorems \ref{printparamgeneral} and \ref{teokappa} imply the next result.

\begin{corollary}
	Let \(s\) be a graphical sequence. The number of connected components has the interval property on $\mathcal{G}(s)$.
\end{corollary}


\subsection{Path-covering number}\label{sectionpathcover}

Let $G$ be a graph. Two paths in $G$ that do not share vertices, are said to be vertex-disjoint. A path covering of $G$ is a set of vertex-disjoint paths of $G$ containing all the vertices of $G$. The path-covering number of $G$, denoted by $\pi(G)$, is the minimum number of paths in a path-covering of $G$. A minimum path-covering in $G$ is a path-covering in $G$ of cardinality $\pi(G)$.

We can look at a path-covering $\mathcal{P}$ of $G$ as a generating forest of $G$, whose components are just the paths of  $\mathcal{P}$. 

\begin{theorem}\label{Teo_Stability_PathCovering}
The path-covering number is stable under 2-switch.
\end{theorem}

\begin{proof}
	Let \(G\) be a graph and $\tau={{a \ b}\choose{c \ d}}$ a 2-switch over \(G\). Suppose that $\mathcal{P}$ is a minimum path-covering in $G$. Then, $\kappa(\mathcal{P})=\pi(G)$. There are three cases:
	\begin{enumerate}
		\item Case 1: \(ab\) and \(cd\) are both edges of the forest \(\mathcal{P}\).
		\item Case 2: \(ab\) and \(cd\) are not edges of the forest \(\mathcal{P}\).
		\item Case 3: one but not both \(ab\) and \(cd\) is an edge of the forest \(\mathcal{P}\).
	\end{enumerate}
	
	Case 1: If \(\tau(\mathcal{P})\) is a forest, then \(\tau(\mathcal{P})\) is a path-covering of \(\tau(G)\). Hence, \(\pi(\tau(G))\leq \pi(G)\). If \(\tau(\mathcal{P})\) is not a forest, then the vertices \(a,b,c,\) and \(d\) are all in the same path \(P\) of \(\mathcal{P}\). This path breaks in a cycle \(C\) and a path \(Q\) after applying \(\tau\). Let \(e\) be an edge of \(C\), \(\tau(\mathcal{P})-e\) is a path-covering of \(\tau(G)\). Hence, \(\pi(\tau(G))\leq \pi(G)+1\).
	
	Case 2: As \(\tau\) is trivial in \(\mathcal{P}\), we have that \(\mathcal{P}\) is a path-covering of \(\tau(G)\). Hence, \(\pi(\tau(G))\leq \pi(G)\).
	
	Case 3: Assume that \(ab \in \mathcal{P}\) and \(cd \notin \mathcal{P}\). Note that \(\tau(\mathcal{P})-ab\) is a path-covering of \(\tau(G)\). Hence, \(\pi(\tau(G))\leq \pi(G)+1\).
	
	Therefore, \(\pi(\tau(G))\leq \pi(G)+1\) and Lemma \ref{lemastable} implies that $\pi$ is stable under 2-switch.
\end{proof}

Theorem \ref{Teo_Stability_PathCovering} together with Theorem \ref{printparamgeneral} implies the next result.

\begin{corollary}
	Let \(s\) be a graphical sequence. The path-covering number has the interval property on $\mathcal{G}(s)$ and on $\mathcal{F}(s)$.
\end{corollary}


\subsection{Chromatic number}\label{sectionchromatic}

A coloring of $G$ is a  function $f:V(G)\rightarrow \mathbb{N}$ such that $x\sim  y$ implies $f(x)\neq f(y)$, for every $xy\in G$. In this context, $f(V(G))$ is called the set of colors of $f$. If $|f(V(G))|=k$, we say that $f$ is a $k$-coloring of $G$. The chromatic number of a graph $G$, denoted with $\chi(G)$, is the smallest value of $k$ for which there is a $k$-coloring in $G$.

As forests are bipartite,  $\chi$ trivially has the interval property on $\mathcal{F}(s)$. This is not so immediate for $\mathcal{G}(s)$.

\begin{theorem}
The chromatic number is stable under 2-switch.
\end{theorem}
\begin{proof}
	 Let $\tau={{a \ b}\choose{c \ d}}$ be a 2-switch over a graph $G$ which has a $k$-coloring $f$, with $k=\chi(G)$.
	 Notice that if $f$ is not a coloring of $\tau(G)$, then either $f(a)=f(c)$ or $f(b)=f(d)$.	
	  
	  Assume without loss that $f(V(G))=\{1,\ldots, k\}$. Define $h:V(G)\rightarrow \mathbb{N}$ as
	 \[
h(g)=\left\lbrace\begin{array}{ll}
f(g),\quad&\text{if }g\not\in\{c,d\},\\
k+1,\quad&\text{otherwise.}
\end{array}\right.
\]	 
Then $h(a)\neq h(c)$ and $h(b)\neq h(d)$, which implies that $h$ is a coloring of $\tau(G)$.
Furthermore, $|h(V(\tau(G))|=k+1$.
Therefore, $\chi(\tau(G)) \leq \chi(G)+1$ and Lemma \ref{lemastable} implies that $\tau$ is stable under 2-switch.
\end{proof}

\begin{corollary}
	Let \(s\) be a graphical sequence. The chromatic number has the interval property on $\mathcal{G}(s)$.
\end{corollary}    
%
%


\subsection{Clique number}\label{sectionclique}

Let $G$ be a graph. A clique of order $i$ in $G$ is a subgraph of $G$ isomorphic to  $K_{i}$, the complete graph on $i$ vertices. The clique number of $G$, denoted by $\omega(G)$, is the maximum order of a clique contained in $G$. 

Since forests are acyclic, $\omega(F)\in\{1,2\}$ for any forest $F$, and hence $\omega$ trivially has the interval property on $\mathcal{F}(s)$. This is not so immediate for $\mathcal{G}(s)$.

\begin{theorem}
The clique number is stable under 2-switch.
\end{theorem}
\begin{proof}
Let $G$ be a graph, $\tau={{a \ b}\choose{c \ d}}$ a a nontrivial 2-switch over $G$ and $K$ a clique of order $\omega(G)$ in $G$.
As $ac,bd\not\in E(G)$, $|E(K)\cap \{ab,cd\}|\leq 1$.

If $ab,cd\not\in E(K)$, then $K$ is clique in $\tau(G)$ and $\omega(\tau(G))\geq \omega(G)$.

Assume $|E(K)\cap\{ab,cd\}|=1$ and assume without loss that $ab\in E(K)$.
As $ac,bd\not\in E(G)$, this implies $c,d\not\in K$. Hence, the only edge in $E(K)-E(\tau(G))$ is $ab$.
Thus $K-a$ is a clique in $\tau(G)$ of order $\omega(G)-1$.  Hence $\omega(\tau(G))\geq \omega(G)-1$.

Therefore, $\omega(\tau(G))\geq \omega(G)-1$ and $\omega$ is stable under 2-switch.
\end{proof}

\begin{corollary}
	Let \(s\) be a graphical sequence. The clique number has the interval property on $\mathcal{G}(s)$.
\end{corollary}


\section{Conclusions}\label{sectionconclusions}
The main results achieved in this work are the following.

\begin{theorem}[Forest Transition Theorem]
		Let $s$ be a graphical sequence and $F,F'\in \mathcal{F}(s)$. Then \(F\) can be transformed into \(F'\) with at most $|E(F')-E(F)|-1$ f-switches given by the Transition Algorithm.
\end{theorem}

\begin{theorem}
	The following parameters are stable under 2-switch:
	\begin{enumerate}
		\item matching number, 
		\item independence number,
		\item domination number,
		\item path-covering number,
		\item edge-covering number,
		\item vertex-covering number,
		\item chromatic number,
		\item clique number,
		\item number of connected components.
	\end{enumerate}
\end{theorem}

\begin{theorem}
	Let \(s\) be a graphical sequence. The following parameters  have the interval property on $\mathcal{G}(s)$ and on $\mathcal{F}(s)$:
	\begin{enumerate}
		\item matching number, 
		\item independence number,
		\item domination number,
		\item path-covering number,
		\item edge-covering number,
		\item vertex-covering number,
		\item chromatic number,
		\item clique number,
		\item number of connected components.
	\end{enumerate}
\end{theorem}
	
Applying Thereom \ref{=s=k} together with Theorem \ref{transbosques}, we get the Transition Theorem for trees.
\begin{theorem}[Tree Transition Theorem]
	Let $T$ and $T'$ be two trees with the same degree sequence. Then \(T\) can be transformed into \(T'\) with at most $|E(T')-E(T)|-1$ t-switches.
\end{theorem}

In the same way that we used the Forest Transition Theorem to prove the interval property
for many parameters in $\mathcal{F}(s)$, analogous versions of the 
result for other families of graphs would yield analogous interval properties.
This raises the question, for which other families of graphs is there a transition theorem?
Bipartite graphs do not have a transition theorem, see the bipartite graphs in
Figure \ref{contrabipartitos}. Notice that both $G_0$ and $G_1$ have degree sequence $s=(6,5,4,4,3,3,3,2,2,2,2)$.
Furthermore, they are not isomorphic as the vertex of degree $6$ in $G_0$ has neighbors of degrees $2,2,2,3,4,5$, whereas the vertex of degree $6$ in $G_1$ has neighbors of degrees $2,2,3,3,4,5$.
Notice that the vertices of degree $4$ are in different partite sets in $G_0$, and in the same partite 
set in $G_1$. 
If $\tau={{a \ b}\choose{c \ d}}$ is non-trivial in $G_0$, then $G_0-\{ab,cd\}$ is connected with 
the same partite sets as $G_0$.
Thus, vertices $3$ and $4$ are in different partite sets in $G_0-\{ab,cd\}$, and if $\tau (G_0)$ is bipartite,
$3$ and $4$ are in different partite sets in $\tau(G_0)$. This will be the case, no matter how many 
2-switches one applies. Hence, to obtain $G_1$, we need to go through a non-bipartite graph at some point.
This means that there is no transition theorem for bipartite graphs.

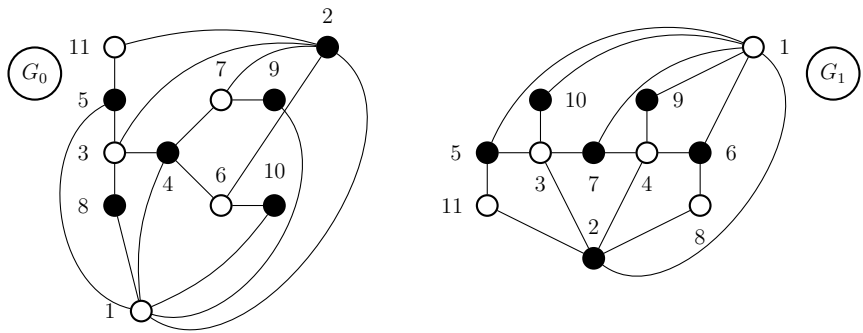
\begin{figure}[H]
	\[
	\begin{tikzpicture}
	[scale=.7,auto=left,every node/.style={scale=.7,circle,thick,draw}] 
	\node (0) at (-1.5,1.5) {$G_{0}$};
	\node [label=left:11](11) at (0,2) {};
	\node [label=left:5,fill](5) at (0,1) {};
	\node [label=left:3](3) at (0,0) {};
	\node [label=left:8,fill](8) at (0,-1) {};
	\node [label=above:9,fill](9) at (3,1) {};
	\node [label=above:10,fill](10) at (3,-1) {};
	\node [label=below:4,fill](4) at (1,0) {};
	\node [label=above:7](7) at (2,1) {};
	\node [label=above:6](6) at (2,-1) {};
	\node [label=above:2,fill](2) at (4,2) {};
	\node [label=left:1](1) at (.5,-3) {};	
	\draw (5) -- (3) ;
		\draw (5) -- (11) ;
	\draw (3) -- (8);
	\draw (9) -- (7) ;
	\draw (6) -- (10) ;
	\draw (3) -- (4) ;
	\draw (7) -- (4) ;
	\draw (6) -- (4) ;
	\draw (2) [bend right=35] to (3) ;
	\draw (2) [bend right=30] to (7) ;	
	\draw (2) -- (6) ;	
		\draw (2) [bend right=15] to(11) ;	
	\draw (1) [bend left=70] to (5);
	\draw (1) -- (8);
	\draw (1) [bend right=75] to (9);
	\draw (1) [bend right=15] to (10);
	\draw (1)[bend left=15] to (4);
	\draw (1) [bend right=90] to (2);
				
	\node (0) at (13.5,1.5) {$G_{1}$};
		\node [label=left:11](11a) at (7,-1) {};
	\node [label=left:5,fill](5a) at (7,0) {};
	\node [label=below:3](3a) at (8,0) {};
	\node [label=below:7,fill](7a) at (9,0) {};
	\node [label=below:4](4a) at (10,0) {};
	\node [label=right:6,fill](6a) at (11,0) {};
	\node [label=below:8](8a) at (11,-1) {};
	\node [label=right:10,fill](10a) at (8,1) {};
	\node [label=right:9,fill](9a) at (10,1) {};
	\node [label=above:2,fill](2a) at (9,-2) {};
	\node [label=right:1](1a) at (12,2) {};	
	\draw (5a) -- (3a) ;
		\draw (5a) -- (11a) ;
	\draw (3a) -- (10a);
	\draw (9a) -- (4a) ;
	\draw (6a) -- (8a) ;
	\draw (3a) -- (7a) ;
	\draw (7a) -- (4a) ;
	\draw (6a) -- (4a) ;
	\draw (2a) -- (3a) ;
		\draw (2a) -- (4a) ;
	\draw (2a) -- (8a) ;	
	\draw (2a)  to (11a) ;	
	\draw (1a)[bend right=30] to (7a) ;	
	\draw (1a)[bend right=45]to (5a);
	\draw (1a)[bend right=30] to (10a);
	\draw (1a) to (9a);
	\draw (1a)  to (6a);
	\draw (1a)[bend left=90]  to (2a);

	\end{tikzpicture}
	\]
	\caption{Two bipartite graphs with the same  degree sequence}
	\label{contrabipartitos}
\end{figure}  

There are many interesting families left for which  degree sequences have been studied,
and for which a transition theorem may exist. One such case are unicyclic graphs, for which we conjecture
there is a transition theorem.

\begin{conjecture}
Given two unicyclic graphs $U,U'$ with the same degree sequence there is a sequence of 2-switches
transforming $U$ into $U'$, such that every intermediate graph is also unicyclic.
\end{conjecture}

\section{Acknowledgements}

This work was partially supported by the Universidad Nacional de San Luis, grant PROICO 03-0918, and MATH AmSud, grant 18-MATH-01.


\end{document}